\documentclass[11pt]{amsart}
\usepackage[left=2cm,top=2cm,right=3cm,nofoot]{geometry}
\usepackage{amsmath,mathtools}%
\usepackage{amsfonts}%
\usepackage{amssymb}%
\usepackage{graphicx}
\usepackage{esint}
\usepackage{hyperref}
\usepackage{enumitem}
\usepackage{accents}
\usepackage{etoolbox}
\patchcmd{\subsection}{-.5em}{.5em}{}{}
\newtheorem{theorem}{Theorem}[section]
\theoremstyle{plain}

\newtheorem{corollary}[theorem]{Corollary}

\newtheorem{definition}[theorem]{Definition}

\newtheorem{lemma}[theorem]{Lemma}

\numberwithin{equation}{section}
\theoremstyle{plain}

\usepackage{etoolbox}
\AtEndEnvironment{proof}{\setcounter{claim}{0}}

\begin{document}

\title[Nodal solutions to Yamabe-type equations]{A note on sign-changing solutions to supercritical Yamabe-type equations}

\begin{abstract}On a closed Riemannian manifold $(M^n ,g)$, we consider the Yamabe-type equation
$-\Delta_g u + \lambda u = \lambda |u|^{q-1}u$, where
$\lambda \in \mathbb{R}_{+}$ and $q>1$. We assume that $M$ admits a proper isoparametric function $f$ with 
focal submanifolds
 of positive dimension. If $k>0$ is the minimum of the dimensions of the focal submanifolds of $f$, we let $q^* =
\frac{n-k+2}{n-k-2}$. We prove the existence of infinite $f$-invariant sign-changing solutions to the equation
when $1<q<q^*$.
\end{abstract}

\author{ Jurgen Julio-Batalla}
\address{ Universidad Industrial de Santander, Carrera 27 calle 9, 680002, Bucaramanga, Santander, Colombia}
\email{ jajuliob@uis.edu.co}
\thanks{  The author was supported by project 3756 of Vicerrector\'ia de Investigaci\'on y Extensi\'on of Universidad Industrial de Santander}

\maketitle

\section{Introduction}

Let $(M^n ,g)$ be a closed connected Riemannian manifold of dimension $n\geq 3$. We let $\Delta =div (\nabla) $ 
be the non-positive Laplace 
operator. In this paper we will  study nodal (sign-changing) solutions of the  Yamabe-type equation 
\begin{equation}\label{YT}
-\Delta_g u + \lambda u = \lambda |u|^{q-1}u ,
\end{equation} 
where $\lambda \in \mathbb{R}_{+}$ and $q>1$.

If the scalar curvature of $g$, ${\textbf{s}}_g$, is constant equal to $\frac{n-2}{4(n-1)} \lambda$ and
$q=p_n =  \frac{n+2}{n-2}$ is the critical Sobolev exponent, then (\ref{YT}) is the Yamabe equation. In this case, if $u$ is a positive solution of the equation, then $u^{\frac{4}{n-2}} g$ is a Riemannian metric, conformal to $g$, which also has constant scalar 
curvature ${\bf s}_g$. It is an important result obtained in several steps by H. Yamabe \cite{Yamabe}, N. Trudinger
\cite{Trudinger}, T. Aubin \cite{Aubin} and R. Schoen \cite{Schoen}, that given any closed Riemannian manifold $(M^n,g)$ of dimension at least 3, there is
a conformal metric of constant scalar curvature. Such conformal metric is unique if it has non-positive
scalar curvature. But in the positive case the solution in general is not unique and many important
results have been obtained about the space of conformal constant scalar curvature metrics, which amounts
to multiplicity results for positive solutions of the Yamabe equation.
Many  results have also been obtained about {\it nodal solutions}  of the Yamabe equation. If $u$
is a nodal solution of Equation (\ref{YT})  then $u$ vanishes somewhere and
$|u|^{\frac{4}{n-2}} g$ is not a Riemannian metric.  However such solutions still have geometric and analytic interest.
In the case of round sphere $(\mathbb{S}^n,g_0)$ the first results about the existence of 
nodal solutions to the Yamabe equation were obtained by Y. Ding \cite{Ding}, using the rich family of 
symmetries of the sphere. Later, more results were obtained for instance by M. del Pino, M. Musso, F. Pacard, A. Pistoia \cite{dpmpp}; M. Medina, M. Musso \cite{Medina} and by J. C. Fern\'andez, J. Petean \cite{JC}. In the context of general
Riemannian manifolds existence of nodal solutions to the Yamabe equation under very general conditions was obtained
by B. Ammann and  E. Humbert \cite{Ammann}. For other results on nodal solutions of this equation see the
articles by G. Henry \cite{Henry}, J. V\'etois \cite{Veto},  J. Julio-Batalla, J. Petean \cite{Nodal} and  M. J. Gursky,  S. P\'erez-Ayala in \cite{Gursky}.

When $q<p_n$, Equation (\ref{YT}) is called subcritical. Solutions of Equation \ref{YT} on the closed Riemannian manifold
$(M^n,g)$ in the subcritical range 
yield solutions of the Yamabe
equation for appropriate manifolds which fiber over $M$ (in particular certain Riemannian products with $M$). 
See for instance the articles by G. Henry, J. Petean \cite{HenryPetean}, by R. Bettiol, P. Piccione
\cite{BettiolPiccione} and by L. L. de Lima, P. Piccione , M. Zedda \cite{deLimaPiccioneZedda} for positive solutions. 
And the article by J. Petean \cite{Petean} for nodal solutions. 
Supercritical equations are more difficult to treat, since
most of the techniques employed in the subcritical and critical cases, do not apply in the supercritical range. But there
are some results, see for instance the article by J. C. Fern\'{a}ndez, O.  Palmas and J. Petean \cite{FernandezPalmasPetean}  
in the case of the sphere.

\medskip

In this article we will consider the situation when the Riemannian manifold $(M^n,g)$ admits a proper isoparametric function $f$  with positive dimensional level sets. Recall that a 
non-constant, smooth  function $f:M\rightarrow[t_0,t_1]$  is called  {\it isoparametric} if there exist smooth functions $a,b$ 
(defined on $[t_0 , t_1 ]$)
which verify that   $|\nabla f|^2=b(f)$ and  $\Delta f=a(f)$.  

From the general theory of isoparametric functions, introduced by Q-M. Wang in \cite{Wang}, 
it is known that the only zeros of the function  $b:[t_0,t_1]\rightarrow \mathbb{R}_{\geq 0}$ are $t_0$ and $t_1$,  which means that the only critical values of $f$ are its minimum and its maximum. It is also known that $M_0 = f^{-1} (t_0 )$ and
$M_1 = f^{-1} (t_1 )$ are smooth submanifolds; they are called the {\it focal submanifolds}
of $f$. We let $d_i$ be  the dimension of $M_i$. The isoparametric function $f$ is called {\it proper} when
$d_0 , d_1 \leq n-2$. In this case, all the level sets 
$ f^{-1} (t)$ are connected. 

We will say that  a function $u$ on $M$ is {\it $f$-invariant} if it is constant along the level sets of $f$, which is 
equivalent to say that $u=\varphi \circ f$ for a function $\varphi : [t_0 , t_1 ] \rightarrow \mathbb{R}$. We
will look for solutions of Equation (\ref{YT}) which are $f$-invariant. 
Many results have been obtained regarding solutions of the Yamabe-type equations invariant by an isoparametric
function. In \cite{HenryPetean} G. Henry and J. Petean considered the Yamabe equation on Riemannian products
with the sphere. They proved multiplicity results for conformal constant scalar curvature metrics on the products by
constructing solutions of subcritical Yamabe-type equations on the sphere invariant under a fixed isoparametric function.
They reduce the equation to an ordinary differential equation and use 
global bifurcation theory. 
In \cite{Henry} G. Henry proved the existence
of a $f$-invariant nodal solution of the Yamabe equation which minimizes energy among nodal solutions. In 
\cite{JC} J. C. Fern\'{a}ndez and J. Petean also use the reduction of the equation to an ordinary differential
equation to show multiplicity results for nodal solutions of the 
Yamabe equation on the round sphere which are  $f$-invariant.  J. C.  Fern\'{a}ndez, O. Palmas and J. Petean proved in \cite{FernandezPalmasPetean} 
multiplicity results for nodal solutions of 
supercritical Yamabe-type equations on the sphere, also using the reduction to an ordinary differential
equation.

\medskip

For a fixed, proper,  isoparametric function $f$ on the closed Riemannian manifold $(M^n,g)$ we let $k=\min \{ d_0 , d_1 \}$, the minimum
of the dimensions of the focal submanifolds. We let $q^* = \frac{n-k+2}{n-k-2}$.

We define the energy of a solution $u$ of Equation (\ref{YT}) as $$E(u)=\int_{M}|u|^{q+1}dv_g,$$
where $dv_g$ is the volume element of $(M^n,g)$.

Our main result is the following

\begin{theorem} Consider a proper isoparametric function $f$ on a closed Riemannian manifold $(M^n,g)$. 
Assume that the minimum dimension of level sets of $f$, $k$, is strictly positive. Assume
that $1<q<\frac{n-k+2}{n-k-2}$. Then
there exist infinite solutions of Equation (\ref{YT}) with arbitrarily large energy.
\end{theorem}

Note that since $p_n = \frac{n+2}{n-2} < \frac{n-k+2}{n-k-2}$, this result applies to critical and supercritical Yamabe-type 
equations. 
It was proved in \cite[Proposition 4.1]{BJP} that there is a $C^0$-bound for positive $f$-invariant solutions to the
Equation (\ref{YT}) when $q<q^*$.  It then follows that only a  finite number of the solutions given by
Theorem 1.1 can be positive solutions, and we have:

\begin{corollary}Consider a proper isoparametric function $f$ on a closed Riemannian manifold $(M^n,g)$. 
Assume that the minimum dimension of level sets of $f$, $k$, is strictly positive. Assume
that $1<q<\frac{n-k+2}{n-k-2}$. Then
there exist infinite nodal solutions of Equation (\ref{YT}).
\end{corollary}

Our result is similar to the result obtained by J. V\'etois in \cite[Corollary 1.1]{Veto}. Indeed, J. V\'etois guarantees 
the existence of large-energy solutions to the Schr\"odinger-Yamabe type equation

$$-\Delta_gu+hu=|u|^{p_n-1}u,$$
where $h$ is a H\"older continuous function, assuming that $h<\frac{n-2}{4(n-1)}\textbf{s}_g$.

We will prove Theorem 1.1 by applying the classical theorem of A. Ambrosetti and P. Rabinowitz 
about the existence of mountain pass critical points  \cite[Theorem 2.13]{Ambrosetti}. This will be carried out in
Section 3. In Section 2 we will recall some information about isoparametric functions and give some preliminary 
results.

\section{Preliminaries on isoparametric functions}

Let $(M^n,g)$ be a closed connected Riemannian manifold. Assume that $M$ admits a proper isoparametric function 
$f:M\rightarrow [ t_0 , t_1 ]$. 
The general theory of isoparametric functions on Riemannian manifolds was introduced by Q-M Wang in
\cite{Wang} following the classical theory, which considered the case of space forms. Recall that the fact that $f$ is an 
isoparametric function means that
there are smooth functions $a$ and $b$ such that $\Delta (f) = a \circ f $ and $\| \nabla f \|^2 = b \circ f$. 
It is proved in \cite{Wang} that the only critical values of $f$ are its minimum and maximum. The critical
level sets, $M_0 = f^{-1} (t_0 )$ and $M_1 = f^{-1} (t_1 )$, are called the {\it focal submanifolds}
of $f$, and it was also
proved in \cite{Wang} that they are actually smooth submanifolds. The regular level sets of $f$, $f^{-1}(t)$  for $t\in
(t_0 , t_1 )$, are called {\it isoparametric hypersurfaces}. We let $d_i$ be the dimension of the focal submanifold  $M_i$
and $k = \min \{ d_0 , d_1 \}$. We will assume  later that 
$k>0$. J.  Ge and  Z.  Tang proved in \cite{Tang} that if $k \leq n-2$ then all the level sets of $f$ are
connected. In this case the isoparametric function $f$ is called {\it proper}.

We consider on $M$  the (normalized) Yamabe-type equation 
\begin{equation}\label{NYT} 
-\Delta_gu +\lambda  u= |u|^{q-1}u,
\end{equation}

where $1<q<(n-k+2)/(n-k-2)$ and $\lambda$ is a positive constant.

In this note we will find solutions of the Yamabe-type equation in the space of $f$-invariant functions.

\begin{definition}
Given an isoparametric function $f$ on the closed manifold $(M^n,g)$ we denote by
$S_f$ the space of $f$-invariant functions on $M$, $S_f = \{ \varphi \circ f : \varphi: [t_0, t_1 ]
\rightarrow \mathbb{R} \}$.
\end{definition}

Let $L^2_1 (M)$ denote the usual Sobolev space of $L^2$-functions on $M$ which admit one derivative also in $L^2(M)$. 
Then  we let $L^2_{1,f}(M) = L^2_1 (M) \cap S_f$ be the space of $L^2_1$-functions on $M$ which are 
$f$-invariant. 

Consider  the functional $I:L^2_{1}(M)\rightarrow\mathbb{R}$ defined as
$$I(u)=\frac{1}{2}\int_{M}\left(|\nabla u|^2+ \lambda  u^2\right)dv_g-  \frac{1}{q+1}\int_M |u|^{q+1}dv_g. $$

Critical points $u$ of the functional $I$ are weak solutions of Equation (\ref{NYT}): they verify

$$D I (u) [v] = \int_{M}(\langle \nabla u , \nabla v \rangle +\lambda u  v ) \ dv_g-   \int_M  u|u|^{q-1} v \ dv_g =0,$$

\noindent
for all $v \in L^1_2 (M)$.

It is important to point out that $L^2_{1,f} (M)$ is a closed subset of $L^2_1 (M)$. Also note that 
if $u = \varphi \circ f \in S_f$ is a $C^2$-function, then $\Delta u = \Delta (\varphi \circ f) = \left(  \varphi ''  b + \varphi '   a 
\right) \circ f
\in S_f$. 

Let $H$ be the $L^2$-orthogonal complement of $L^2_{1,f} (M)$ in $L^2_1 (M)$. Then for any $v\in H$ and any
$u\in L^2_{1,f} (M) \cap C^2 (M)$ we have that 

$$ \int_{M}\langle \nabla u , \nabla v \rangle dv_g =  - \int_{M}\ (\Delta u )v \ dv_g =0 .$$ 

Since $L^2_{1,f} (M) \cap C^2 (M)$ is dense in $L^2_{1,f} (M)$ it follows that for any $u \in L^2_{1,f} (M)$ and
any $v\in H$ we have that 

$$D I (u) [v] = \int_{M}(\langle \nabla u , \nabla v \rangle +\lambda uv ) \ dv_g-  \int_M  u|u|^{q-1} v \ dv_g =0.$$

Therefore we have:

\begin{lemma}
If $u\in L^2_{1,f} (M)$ is a critical point of the restriction of $I$ of $L^2_{1,f} (M)$, $I: L^2_{1,f} (M) \rightarrow \mathbb{R}$, 
then $u$ is a strong solution of Equation (\ref{NYT}).
\end{lemma} 

The previous comments show that if $u\in L^2_{1,f} (M)$ is a critical point of the restriction of $I$ of $L^2_{1,f} (M)$, then
$u$ is a critical point in $L^2_1 (M)$, and therefore a weak solution of Equation (\ref{NYT}). Then it follows by
classical elliptic regularity that $u$ is a strong solution.

\medskip

It was proved in \cite[Section 6]{Henry} that if $1<q<(n-k+2)/(n-k-2)$ then $L^2_{1,f} (M)\subset L_f^{q+1}(M)    $
and moreover the inclusion is compact. The following lemma is part of the previous statement, we include a short proof for
completeness since we will need it in the next section:

\begin{lemma} Assume that  $1<q<(n-k+2)/(n-k-2)$. There exists a positive constant $C$ such that 
if $u\in  L^2_{1,f} (M)$, then 

$$|u|_{L^{q+1}(M)}\leq C |u|_{L^2_1(M)}.$$

\end{lemma}

\begin{proof}
Since the focal submanifolds $M_0,M_1$ of the function $f$ have positive dimensions ($d_0,d_1$, respectively), 
it is well-known that $M^n$ can be identified as a union of two disk bundles $D_0, D_1$, each one over $M_0$ and $M_1$ respectively. More precisely,  for $i\in\{0,1\}$ let $exp_{M_i}$ be the normal exponential map of $M_i$ in $M$. Following the notation
used by R. Miyaoka  in \cite{Miyaoka} the manifold $M^n$ is diffeomorphic to 
$$N_{\leq a}M_0\cup N_{\leq a}M_1,$$
where $$N_{\leq a}Q=\{exp_{Q}(s\eta)/\quad|\eta|=1, s<a\}\cup \{\exp_{Q}(a\eta)/\quad|\eta|=1\},$$
and $2a=d_g(M_0,M_1)$.

From the  definition of disk bundle, we can choose a coordinate system $(U_j,\varphi_j)$ on $N_{\leq a}M_0$ (and similar ones  on $N_{\leq a}M_1$) such that  $U_j=\pi^{-1}(U'_j)$ for a finite cover $\{U'_j\}$ of $M_0$ and each $\varphi_j$ is a diffeomorphism defined  by $$\varphi_j:U_j\rightarrow U'_j\times D^{n-m_0}_0(a),$$ where $D^{n-m_0}_0(a)$ is the disk of radius $a$ in the normal bundle of $M_0$.

Without loss of generality, we cover $M^n$ by a finite number $m$ of these type of charts with a uniform bound for the metric tensor $g$ i.e. there exists a constant $c>1$ with $$c^{-1} \ Id \leq g^l_{ij}\leq c \ Id\quad\text{for}\quad l\in\{1,\cdots,m\}.$$

Let $s=n-d_0$.

For any $f$-invariant function  $u$ we have that $u$ only depends on the  normal coordinates on $N_{\leq a}M_0$ and $N_{\leq a}M_1$. In particular, using the previous charts $(U_j,\varphi_j)$ on $N_{\leq a}M_0$, the function $u$ only depends on $D^s_0(a)$ (similarly, $u$ only depends on $D_0^{n-d_1}(a)$ in $N_{\leq a}M_1$).

For some positive constants $k_0$ and $k_1$ we have:
\begin{eqnarray*}
\int_{U_l} u^{q+1} dv_g&=\int\limits_{U'_l\times D_0^s(a)}u^{q+1}\sqrt{det g^l_{ij}}dxdy\\
&=k_0 \int\limits_{D_0^s(a)}u^{q+1}\sqrt{det g^l_{ij}}dy\\
&\leq k_1\int\limits_{D_0^s(a)}u^{q+1}dy.
\end{eqnarray*}
Therefore for some positive constant $k_2$ we have: $$|u|_{L^{q+1}(U_l)}\leq k_2|u|_{L^{q+1}(D_0^s(a) )}.$$
Since $g^l_{ij}$ is also bounded from below, with a similar argument we can get that:
$$|u|_{L^2_1(U_l)}\geq k_3 |u|_{L^2_1(D_0^s(a))},$$
for some positive constant $k_3$.

Applying the usual Sobolev inequalities in Euclidean space  $D_0^s(a)\subset\mathbb{R}^{n-m_0}$ we have
that for some positive constant $k_4$ we have:   $$|u|_{L^{q+1}(D_0^s(a)) }\leq k_4 |u|_{L^2_1(D_0^s(a))},$$
and therefore, $$|u|_{L^{q+1}(U_l)}\leq C(l) |u|_{L^2_1(U_l)}. $$

Then the inequality in the lemma follows by taking a partition of unity on $M$ subordinate to $\{U_l\}_l$ and adding up.

\end{proof}

\section{Mountain pass critical points}

In this section we will give the proof of Theorem 1.1.
We will apply the classical min-max method of Ambrosetti-Rabinowitz to the functional $I:L^2_{1,f}(M)\rightarrow\mathbb{R}$ defined as
$$I(u)=\frac{1}{2}\int_{M}\left(|\nabla u|^2+\lambda u^2\right)dv_g-\frac{1}{q+1}\int_M |u|^{q+1}dv_g. $$

We briefly recall the min-max Theorem 2.13 in \cite{Ambrosetti}:

Let $B(\rho)$ be the open ball of center at $0$  and radius $\rho>0$ in the space $L_{1,f}^2(M)$ and 
$$\Gamma^*=\{h:L_{1,f}^2(M)\rightarrow L_{1,f}^2(M)/ \;I(h(B(1))\geq 0\;,\;h\;\text{is an odd  homeomorphism}\}.$$

If  the functional $I$ satisfies the conditions
\begin{enumerate}
\item \label{c1}$I(0)=0$ and there exist $\rho, \delta >0$ such that $I>0 $ in $B(\rho)-\{0\}$ and $I\geq \delta $ on $\partial B(\rho);$
\item \label{c2} the Palais-Smale condition: if $u_m$ is a sequence in $L^2_{1,f}(M)$ such that $I(u_m )$ is bounded and $DI (u_m ) \rightarrow 0$, then it has a convergent subsequence;
\item \label{c3} $I(u)=I(-u)$ for any $u \in L^2_{1,f}(M)$, and for any finite dimensional subspace $E$ in $L_{1,f}^2(M)$, $E\cap \{I\geq0\}$ is bounded.

\end{enumerate}
Then the content of Theorem 2.13 in \cite{Ambrosetti} is that  given a sequence $E_m$ of subspaces of $E$,  $E_m \subset
E_{m+1}$  ($m$ = dim($E_m$)), there exists an increasing sequence of critical values $c_m$ of $I$, given by 
$$c_m=\sup\limits_{h\in\Gamma^*}\inf\limits_{u\in E_{m-1}^{\perp}\cap\; \partial B(1)}I(h(u)).$$

\medskip

First we will verify that $I$ in fact satisfies conditions  \ref{c1}, \ref{c2} and \ref{c3}. 

Indeed,

\medskip

-Condition \ref{c1}:

Note that since  $\lambda>0$, the operator $P:=-\Delta +\lambda$ is coercive, and therefore there exists a positive constant $C_1$ such that 
for any function $u\in L_{1,f}^2 (M)$ we have:
$$\int_{M} u P(u)dv_g\geq C_1 |u|^2_{L_{1}^2(M)}.$$

Moreover, by Lemma 2.3 there exists a positive constant $C_2$ such that for any function $u\in L_{1,f}^2 (M)$ we have:
$$\int_{M} u^{q+1}dv_g\leq C_2 |u|^{q+1}_{L^2_{1}(M)}.$$

Hence $$I(u)\geq C_3 |u|^2_{L_{1}^2(M)}- C_4 |u|^{q+1}_{L^2_{1}(M)},$$
for positive constants $C_3 ,C_4 $.

Consider a function $u \in L^2_{1,f} (M)$ with $|u|_{L^2_1(M)}=1$. It follows from the previous inequality that for any $t>0$, 
$I(tu) \geq t^2 C_3 - t^{q+1} C_4$. Then we can deduce that the  assumption that $q+1>2$ implies the existence of a small $\rho>0$ such that the functional $I$ satisfies  Condition \ref{c1}.

\medskip

-Condition \ref{c2}: 

Consider a Palais-Smale sequence $u_m$  in $L^2_{1,f}(M)$ for the functional $I$. Then it follows that the sequence is
bounded in $L^2_{1,f}(M)$, and therefore it has a subsequence which is weakly convergent. Then by
the compactness of the embedding of  $L^2_{1,f}(M)$ in $ L_f^{q+1}(M)$ \cite[Lemma 6.1]{Henry} we have  a convergent subsequence.

\medskip

-Condition \ref{c3}:

Let $E_m$ be a subspace of $L^2_{1,f}(M)$ of dimension $m$. There is a positive  constant $C$ 
so that for all $u$ in $E_m$ we have  $$\int_{M} uP(u)dv_g\leq C|u|^{q+1}_{L^{q+1}(M)}.$$

Then for any  $u \in E_m$ with $\int_M uP(u)dv_g=1$ and for $t>0$ we have $$I(tu)\leq \frac{t^2}{2}-\frac{t^{q+1}}{(q+1)C}.$$

Thus $E_m\cap\{I\geq0\}$ is bounded.

\bigskip

Therefore we can apply Theorem 2.13 in \cite{Ambrosetti}  to obtain that there exists a sequence 
$u_m\in L^2_{1,f}$ of  critical points of the functional $I:L^2_{1,f}(M)\rightarrow\mathbb{R}$,
associated to the mountain pass level $c_m$. The
critical points are strong solutions of the Yamabe-type equation (\ref{YT}) by Lemma 2.2.

\medskip

It follows that to conclude the proof of Theorem 1.1 we only need to prove that
the solutions $u_m$ have large energy, i. e. that the increasing sequence $c_m$ goes to infinity.

Let $\{ e_i \}_{i\geq 1}$ be an orthonormal basis of $ L^2_{1,f}(M)$ and $E_m = \langle e_1 , ...,e_m \rangle$.

Let $$N=\left\lbrace u\in L^2_{1,f}(M)-\{0\}\;/\quad \frac{1}{2}\int_{M} uP(u)dv_g=\frac{1}{q+1}\int_{M} |u|^{q+1}dv_g \right\rbrace,$$
and consider 
$$d_m=\inf\left\lbrace \left(\int_{M} uP(u)dv_g\right)^{1/2} \;/\; u\in N \cap E_m^{\perp}\right\rbrace.$$

Clearly $d_m$ is a non-decreasing sequence. We will show that $d_m \rightarrow \infty$.
Assume that  $(d_m)$ is bounded.  Then there exists $v_m\in N\cap E_m^{\perp}$ and a positive constant $d$ such that  $$\int_{M} v_mP(v_m)dv_g\leq d,\quad \forall m.$$

Since $P$ is coercive it follows that the sequence $v_m$ is bounded in $L^2_{1,f}(M)$ and therefore it is weakly convergent to $v\in L^2_{1,f}(M)$. 
Since $v$ is orthogonal to every $E_m$ we must have that $v=0$.
Since the inclusion  $L^2_{1,f}(M)\subset L_f^{q+1}(M)$ is compact, a subsequence satisfies that
$v_m\rightarrow 0$ strongly in $L^{q+1}_f(M)$.

Also by Lemma 2.3 there exists a constant  $K>0$ such that
$$K<\dfrac{|v_m|_{L^2_{1}(M)}}{|v_m|_{L^{q+1}(M)}}.$$
Using that $P$ is coercive and the definition of $N$ we have that there exists a positive constant $K_0$ such that
$$\dfrac{|v_m|_{L^2_{1}(M)}}{|v_m|_{L^{q+1}(M)}} \leq K_0 \dfrac{\left(\int_M v_mP(v_m)dv_g\right)^{1/2}}{\left(\int_M v_mP(v_m)dv_g\right)^{1/(q+1)}}.$$
Then since  $q+1>2$ it follows that $$\int_M v_mP(v_m)dv_g$$ 
is bounded from below by a positive constant. Again by definition of $N$, the sequence $$|v_m|_{L^{q+1}(M)}$$ also must be bounded from below by a positive constant. This is a contradiction. 

Therefore $d_m\rightarrow\infty$.

Now, following  a similar approach used in \cite{Ambrosetti} we will compare the values $d_m$ and $c_{m}$:

First note that for all $u\in L_{1,f}^2(M)-\{0\}$ there exists a constant $C>1$ such that $$\left(\int_{M}uP(u)dv_g\right)^{1/2}< C|u|_{L^2_1(M)}.$$

Moreover  we can find a unique positive constant $\alpha(u)$ such that $\alpha(u)u\in N$. In particular  
\begin{equation}\label{condition-N}
\frac{1}{2(\alpha(u))^{q-1}}\int_M uP(u)dv_g=\frac{1}{q+1}\int_M |u|^{q+1}dv_g.
\end{equation}

Additionally, if $u\in E_m^{\perp}$ then 
\begin{equation}\label{condition-dm}
d_m\leq \left(\int_M\alpha uP(\alpha u)dv_g\right)^{1/2}< \alpha C |u|_{L^2_1(M)}.
\end{equation}

These constants are fundamental in the next construction

\begin{lemma}{ \cite{Ambrosetti}}
For each $m$ there exists $h_m\in \Gamma^*$ such that $$\inf\limits_{u\in E_m^{\perp}\cap \partial B(1)}I(h_m(u))\geq kd_m^2$$
for some constant $k>0.$  
\end{lemma}

\begin{proof}
Using \ref{condition-N},\ref{condition-dm} we have that, for each $u\in E_m^{\perp}$ with $|u|_{L^2_1(M)}\leq 1$,
\begin{align*}
I\left(\frac{d_m}{C}u\right)=&\frac{d_m^2}{2C^2}\int_M uP(u)dv_g-\left({\frac{d_m}{C}}\right)^{q+1}\frac{1}{2\alpha^{q-1}}\int_M uP(u)dv_g\\
=&\frac{d_m^2}{C^2}\left(1-\left(\frac{d_m}{C\alpha}\right)^{q-1}\right)\left(\frac{1}{2}\int_M uP(u)dv_g\right)\\
\geq& d_m^2k_0\int_M uP(u)dv_g
\end{align*}
for some $k_0>0$. 

Thus 
\begin{equation}\label{A}
\frac{d_m}{C}( E_m^{\perp}\cap B(1))\subset \{I> 0\}\cup\{0\}.
\end{equation}

Furthermore, from the coercivity of $P$ it follows that $$I\left(\frac{d_m}{C}u\right)\geq kd_m^2,$$
for some $k>0$ and for all $u\in E_m^{\perp}\cap \partial B(1)$.

We can suppose that there is some $\epsilon>0$ for which 
\begin{equation}\label{B}
Z_{\epsilon}:=\frac{d_m}{C}( E_m^{\perp}\cap B(1))\bigoplus\epsilon(E_m\cap B(1))\subset \{I>0\}\cup \{0\}.
\end{equation}

Indeed, if it is not the case, there exist $\epsilon_i$ and $u_i\in Z_{\epsilon_i}-\{0\}$ such that $\epsilon_i\rightarrow0$ and $u_i\notin\{I>0\}$. The sequence $(u_i)$ is bounded in $L_{1,f}^2(M)$, so as we discussed above $u_i$ converges weakly in $L_{1,f}^2(M)$ and strongly in $L^{q+1}_f(M)$ to $u_0$. Again by Lemma 2.3 and coercivity of $P$,
$$0<K<\dfrac{\int_M u_iP(u_i)dv_g}{|u_i|^2_{L^{q+1}(M)}}.$$
Since $I(u_i)\leq 0$ and $q>1$ it follows that $|u_i|_{L^{q+1}(M)}$ is uniformly bounded from below and therefore $\int_Mu_0^{q+1}dv_g>0$. In particular $u_0\neq 0$ and it follows from \ref{A} that $I(u_0)>0$. On the other hand since $u_0$ is the weak limit in $L^2_{1,f}(M)$ of $u_i$, we have that $I(u_0)\leq 0$. This is a contradiction and therefore we can pick a positive $\epsilon$ such that \ref{B}.

We can fix such $\epsilon$ and define the linear map $h_m$ from $L^2_{1,f}(M)$ into $L^2_{1,f}(M)$ as $$h_m(u):=\frac{d_m}{C}u_1 + \epsilon u_2,$$
where $u=u_1+u_2$ for $u_1\in E_m^{\perp}$, $u_2\in E_m$.

By construction $h_m(B(1))\subset Z_{\epsilon}\subset\{I\geq 0\}$. Then $h_m\in \Gamma^*$.

Finally, $$\inf\limits_{u\in E_m^{\perp}\cap \partial B(1)}I(h_m(u))\geq kd_m^2.$$

\end{proof}

It is clear from the definition of $c_m$ and Lemma 3.1 that $c_{m+1}\geq kd_m^2$ and hence the proof of Theorem 1.1 is complete.

\end{document}